\documentclass[a4paper,10pt]{article}

\usepackage[utf8]{inputenc} 
\usepackage{color}
\usepackage[ruled, vlined]{algorithm2e}
\usepackage{relsize}
\usepackage{bbold}

\makeatletter
\newcommand{\LoadPackagesNow}{}
\newcommand{\LoadPackageLater}[1]{%
   \g@addto@macro{\LoadPackagesNow}{%
      \usepackage{#1}%
   }%
}
\makeatother

\usepackage{xspace}
\usepackage{xargs}
\usepackage{ifthen}
\usepackage{etoolbox}

\usepackage[
	table 
]{xcolor}
\usepackage[%
]{graphicx}
\usepackage{float}
\usepackage{wrapfig}

\usepackage{subfigure}
\usepackage{caption}
\usepackage{multicol}
\usepackage[
   tbtags,    
   sumlimits,  
   nointlimits, 
   namelimits, 
   reqno,     
]{amsmath} %

\usepackage{amsfonts}
\usepackage{mathrsfs} 
\usepackage{dsfont}
\usepackage{amssymb}
\LoadPackageLater{amsthm}
\LoadPackageLater{thmtools}

\usepackage{geometry}
\usepackage[%
	square,	
	comma,	
	numbers,	
	sort,		
	sort&compress,    
]{natbib}

\usepackage{framed}
\usepackage[normalem]{ulem}      
\usepackage{soul}		            
\usepackage{url} 
\usepackage{lipsum}
\usepackage[textsize=tiny,english,colorinlistoftodos,disable]{todonotes}

\usepackage{enumitem}
\definecolor{pdfurlcolor}{rgb}{0,0,0.6}
\definecolor{pdffilecolor}{rgb}{0.7,0,0}
\definecolor{pdflinkcolor}{rgb}{0,0,0.6}
\definecolor{pdfcitecolor}{rgb}{0,0,0.6}
\usepackage[
  colorlinks=true,         
  urlcolor=pdfurlcolor,    
  filecolor=pdffilecolor,  
  linkcolor=pdflinkcolor,  
  citecolor=pdfcitecolor,  %
  raiselinks=true,			 
  breaklinks,              
  verbose,
  hyperindex=true,         
  linktocpage=true,        
  hyperfootnotes=false,     
  bookmarks=true,          
  bookmarksopenlevel=1,    
  bookmarksopen=true,      
  bookmarksnumbered=true,  
  bookmarkstype=toc,       
  plainpages=false,        
  pageanchor=true,         
  pdftitle={Compressed Sensing for Finite-Valued Signals},             
  pdfauthor={Sandra Keiper},            
  pdfcreator={LaTeX, hyperref, KOMA-Script}, 
  pdfdisplaydoctitle=true, 
  pdfstartview=FitH,       
  pdfpagemode=UseOutlines, 
  pdfpagelabels=true,           
  pdfpagelayout=SinglePage, 
]{hyperref}

\LoadPackagesNow

\usepackage{bm}

\makeatletter
\g@addto@macro\bfseries{\boldmath}
\makeatother

\newcommand{\ifargdef}[3][{}]{\ifthenelse{\equal{#2}{}}{#1}{#3}}

\newtheoremstyle{claim}
	{\topsep}{\topsep}%
	{\itshape}
	{}
	{\bfseries\boldmath}
	{}
	{.5em}
	{\thmname{#1} \thmnumber{#2} \thmnote{(#3)}}
\newtheoremstyle{definition}
	{\topsep}{\topsep}%
	{}
	{}
	{}
	{}
	{.5em}
	{{\bfseries\thmname{#1} \thmnumber{#2}} \thmnote{(#3)}}
\newtheoremstyle{remark}
	{\topsep}{\topsep}%
	{}
	{}
	{}
	{}
	{.5em}
	{{\itshape\thmname{#1}:}}
\newtheoremstyle{example}
	{\topsep}{\topsep}%
	{}
	{}
	{}
	{}
	{.5em}
	{{\itshape\thmname{#1}:}}

\declaretheorem[style=claim,numberwithin=section]{theorem}
\declaretheorem[style=claim,sibling=theorem]{proposition}
\declaretheorem[style=claim,sibling=theorem]{lemma}

\declaretheorem[style=claim,sibling=theorem]{remark}

\declaretheorem[style=definition,sibling=theorem]{definition}

\newcommand{\opleft}[1]{\mathopen{}\left#1}
\newcommand{\opright}[1]{\right#1\mathclose{}}
\newcommandx{\braces}[4]{%
\ifstrequal{#3}{normal}{#1#4#2}{%
\ifstrequal{#3}{auto}{\left#1#4\right#2}{%
\ifstrequal{#3}{opauto}{\opleft#1#4\opright#2}{%
#3#1#4#3#2}}}%
}
\newcommandx{\opannot}[3][3=\downarrow]{\stackrel{\mathclap{\substack{#1 \\ #3 \vspace{2pt}}}}{#2}}
\newcommandx{\lineannot}[3][3=\rightarrow]{\mathllap{\boxed{\text{\textsmaller{#1}}} #3} #2}
\newcommandx{\multilineannot}[4][4=\rightarrow]{\mathllap{\boxed{\parbox{#1}{\RaggedRight\textsmaller{#2}}} #4} #3}
 
\newcommand{\N}{\mathbb{N}} 
\newcommand{\R}{\mathbb{R}} 
 
\newcommand{\sprod}[1]{\left\langle #1 \right\rangle}
\DeclareMathOperator{\cone}{cone}
\DeclareMathOperator{\one}{\mathbb{1}}

\newcommand{\prb}[1]{\mathbb{P}\left( #1 \right)}
\newcommand{\erw}[1]{\mathbb{E}\left( #1 \right)}
\newcommand{\sse}{\subseteq}

\newcommand{\suchthat}[1][normal]{\ifstrequal{#1}{normal}{\mid}{#1|}} 

\newcommandx{\intvcl}[3][1=normal]{\braces{[}{]}{#1}{#2, #3}} 
\newcommandx{\intvop}[3][1=normal]{\braces{(}{)}{#1}{#2, #3}} 
\newcommandx{\intvclop}[3][1=normal]{\braces{[}{)}{#1}{#2, #3}} 
\newcommandx{\intvopcl}[3][1=normal]{\braces{(}{]}{#1}{#2, #3}} 

\newcommand{\chull}{\operatorname{conv}}

\DeclareMathOperator*{\argmin}{argmin} 
\newcommandx{\abs}[2][1=normal]{\braces{\lvert}{\rvert}{#1}{#2}} 
\newcommandx{\ceil}[2][1=normal]{\braces{\lceil}{\rceil}{#1}{#2}} 
\newcommandx{\floor}[2][1=normal]{\braces{\lfloor}{\rfloor}{#1}{#2}} 
\newcommandx{\round}[2][1=normal]{\braces{[}{]}{#1}{#2}} 
\newcommandx{\der}[1]{D^{#1}} 
\newcommandx{\partder}[4][1={},4={}]{\frac{\partial^{#4} #2}{\partial #3^{#4}}\ifargdef{#1}{\Big|_{#1}}} 
\newcommandx{\integ}[4][1={},2={}]{\int_{#1}^{#2} #3 \, #4} 
\newcommandx{\asympffaster}[2][1=normal]{o\braces{(}{)}{#1}{#2}} 
\newcommandx{\asympfaster}[2][1=normal]{O\braces{(}{)}{#1}{#2}} 
\newcommandx{\asympeq}[2][1=normal]{\Theta\braces{(}{)}{#1}{#2}} 
\newcommandx{\asympsslower}[2][1=normal]{\omega\braces{(}{)}{#1}{#2}} 
\newcommandx{\asympslower}[2][1=normal]{\Omega\braces{(}{)}{#1}{#2}} 

\newcommandx{\norm}[2][1=normal]{\braces{\|}{\|}{#1}{#2}} 
\renewcommandx{\sp}[3][1=normal]{\braces{\langle}{\rangle}{#1}{#2, #3}} 
\newcommand{\adj}[1]{{#1}^\ast} 
\newcommandx{\End}[2][2={}]{\mathcal{L}\opleft( #1 \ifargdef{#2}{, #2} \opright)} 
\newcommandx{\measure}[2][1=normal]{\operatorname{vol}\braces{(}{)}{#1}{#2}} 
\DeclareMathOperator{\supp}{supp} 
\newcommandx{\Leb}[3][1={},3=normal]{L^{#2}\ifargdef{#1}{\braces{(}{)}{#3}{#1}}{}} 
\newcommandx{\Lebnorm}[4][1=normal,3={2},4={}]{\norm[#1]{#2}_{\Leb[#4]{#3}}} 
\renewcommandx{\l}[3][1={},3=normal]{\ell^{#2}\ifargdef{#1}{\braces{(}{)}{#3}{#1}}} 
\newcommandx{\lnorm}[4][1=normal,3={2},4={}]{\norm[#1]{#2}_{\l[#4]{#3}}} 
\newcommandx{\Smooth}[4][1={},3={},4=normal]{C_{#3}^{#2}\ifargdef{#1}{\braces{(}{)}{#4}{#1}}} 
\newcommandx{\Schwartz}[2][1={},2=normal]{\mathscr{S}\ifargdef{#1}{\braces{(}{)}{#2}{#1}}} 
\newcommandx{\Schwartzpoly}[2][1=normal]{\braces{\langle}{\rangle}{#1}{\abs[#1]{#2}} } 
\newcommandx{\Tempdistr}[2][1={},2=normal]{\mathscr{S}'\ifargdef{#1}{\braces{(}{)}{#2}{#1}}} 
\newcommandx{\distrinp}[3][1=normal]{\braces{\langle}{\rangle}{#1}{#2, #3}} 
\newcommand{\Linedistr}[1][]{\mathfrak{L}\ifargdef{#1}{_{#1}}{}} 
\newcommandx{\ft}[3][1=default,2=auto]{
\ifstrequal{#1}{default}{\widehat{#3}}{
\ifstrequal{#1}{long}{{\braces{(}{)}{#2}{#3}}^{\wedge}}{}}} 

\newcommand{\calC}{\mathcal{C}}

\newcommand{\calP}{\mathcal{P}}

\newcommandx{\ift}[3][1=default,2=auto]{
\ifstrequal{#1}{default}{\check{#3}}{
\ifstrequal{#1}{long}{{\braces{(}{)}{#2}{#3}}^{\vee}}{}}} 

\newcommand{\rnd}{\operatorname{round}}

\newcommand{\set}[1]{\left\{ #1 \right\}} 
\graphicspath{{images/}} 

\newcommand{\sk}[1]{{\color{black}{#1}}}

\newcommand{\sg}[1]{{\color{black}{#1}}}
\usepackage[normalem]{ulem}

\title{\bfseries Recovery of Binary Sparse Signals with Biased Measurement Matrices}
\author{\hspace*{-1cm}
  Axel Flinth, Sandra Keiper\\[.5em]
  \small{\textsc{\hspace*{-1cm}Institut für Mathematik, Technische Universit\"at Berlin }}\\[.5em]
}

\begin{document}

\listoftodos
\newcommand{\OpAnalysis}[1]{T_{#1}} 
\newcommand{\OpSynthesis}[1]{\adj T_{#1}} 
\newcommand{\OpFrame}[1]{S_{#1}} 
\newcommand{\defsf}{\varphi} 
\newcommand{\InpSp}{\mathcal{H}} 
\newcommand{\InpSpK}{\InpSp_K} 
\newcommand{\ProK}{P_K} 
\newcommand{\InpSpM}{\InpSp_M} 
\newcommand{\ProM}{P_M} 
\newcommand{\sig}{x^0} 
\newcommand{\sigrec}{x^\star} 
\newcommand{\PF}{\Phi} 
\newcommand{\pf}{\phi} 
\newcommand{\cluster}{\Lambda} 
\newcommand{\concentr}[2]{\kappa\ifargdef{#1}{\opleft( #1, #2 \opright)}} 
\newcommand{\clustercoh}[2]{\mu_c \ifargdef{#1}{( #1 , #2)}} 
\newcommand{\anorm}[2]{\norm{#1}_{1,#2}} 
\newcommand{\ver}{\mathrm{v}} 
\newcommand{\hor}{\mathrm{h}} 
\newcommand{\dir}{\imath} 
\newcommand{\meyerscal}{\phi} 
\newcommand{\Scalfunc}{\Phi} 
\newcommand{\Corofunc}{W} 
\newcommand{\Coro}{\mathscr{K}} 
\newcommand{\conefunc}{v} 
\newcommand{\Conefunc}[1]{V_{(#1)}} 
\newcommand{\Cone}[1]{\mathscr{C}_{(#1)}} 
\newcommand{\pscal}{A} 
\newcommand{\pshear}{S} 
\newcommand{\pscalcone}[2]{A_{#1,(#2)}} 
\newcommand{\shearcone}[1]{S_{(#1)}} 
\newcommand{\unishplain}{\psi} 
\newcommand{\unishplainft}{\ft{\unishplain}} 
\newcommand{\unish}[5][{}]{\unishplain_{#3,#4,#5}^{#2\ifargdef{#1}{,(#1)}}} 
\newcommand{\unishft}[5][{}]{\unishplainft_{#3,#4,#5}^{#2\ifargdef{#1}{,(#1)}}} 
\newcommand{\Scalparamdomain}{\mathsf{A}} 
\newcommand{\aj}{\alpha_j} 
\newcommandx{\Unish}[3][1=\meyerscal,2=\conefunc,3=(\aj)_j]{\operatorname{SH}(#1, #2, #3)} 
\newcommandx{\UnishLow}[1][1=\meyerscal]{\operatorname{SH}_{\mathrm{Low}}(#1)} 
\newcommandx{\UnishInt}[3][1=\meyerscal,2=\conefunc,3=(\aj)_j]{\operatorname{SH}_{\mathrm{Int}}(#1, #2, #3)} 
\newcommandx{\UnishBound}[3][1=\meyerscal,2=\conefunc,3=(\aj)_j]{\operatorname{SH}_{\mathrm{Bound}}(#1, #2, #3)} 
\newcommand{\Unishgroup}{\Gamma} 
\newcommand{\Unishind}{\gamma} 
\newcommand{\lmax}{l_j} 
\newcommand{\weight}{w} 
\newcommand{\wlen}{\rho} 
\newcommand{\model}{{\weight\Linedistr}} 
\newcommand{\modelrec}{\model^\star} 
\newcommand{\Corofilter}{F} 
\newcommand{\mdiam}{h} 
\newcommand{\mask}[1]{\mathscr{M}_{#1}} 
\newcommand{\Unishshort}{\Psi} 
\newcommand{\scalpm}[1]{#1^{\pm1}} 
\newcommand{\translind}[2]{#1^{(#2)}} 
\maketitle

\begin{abstract}
This work treats the recovery of sparse, binary signals through box-constrained basis pursuit using biased measurement matrices. Using a probabilistic model, we provide conditions under which the recovery of both sparse and \sg{saturated} binary signals is \sg{very likely}. In fact, we \sg{also} show that under the same condition, the solution of the boxed-constrained basis pursuit program can be found using boxed-constrained least-squares.
\end{abstract}

\vspace{.1in}
\vspace{.1in}

\section{Introduction}

Compressed sensing is nowadays a well-known effective tool to aquire signals $x_0\in \R^N$ from underdetermined systems of linear equations $Ax_0+n=b\in \R^m$, where $n\in \R^m$ is some noise vector. If $m< N$, the recovery problem is per se ill-posed, but can be turned into a well-posed one by imposing an a-priori structure on $x_0$.  An important choice, which has been treated extensively in the literature, is that of \emph{sparsity}, i.e., that only a few entries in $x_0$ are different from zero. The framework of \emph{compressed sensing} offers a systematic way of analyzing such inverse problems. \sg{We refer to \cite{DDEK} for a} survey of the most important results.

Another structural assumption of interest is that of $x_0$ having values in a finite discrete alphabet. \sg{Finite-valued and sparse signals appear, for example, in error correcting codes \cite{CRTV} as well as massive Multiple-Input Multiple-Output (MIMO) channel \cite {RHE} and wideband spectrum sensing \cite{ALLP}. A particular example is given by wireless communications, where the transmitted signals are sequences of bits, i.e., with entries in $\{0,1\}^N$ or $\{-1,1\}^N$. One could for instance think of $x_0$ as a representation of certain transmitters being either on (\sk{$(x_0)_i=1$}) or off (\sk{$(x_0)_i=0$}). The operator $A$ then models the map from transmitter configurations to measurements at a receiver.
In this \sk{work}, we will concentrate on the case of binary signals with $x_0 \in \{0,1\}^N$. However, all the results hold also true for binary signals with $x_0 \in \{-1,1\}^N$ or any other binary dictionary through translation.}

 In certain types of networks just described, it is reasonable to assume that only a few transmitters are active at a certain instance, which naturally induces sparsity. Hence, it is interesting to consider signals which enjoy \emph{both} structures at the same time. This problem has only very recently been considered in the literature. We refer to\cite{KKLP} for an introduction, as well as for a literature review.

 In the \sg{remainder of the introduction}, we specifically aim to review a small subselection of the known results, in particular a few which we will need in the following to keep the paper self-contained.

%

\sk{\subsection{Random Measurement Matrices}
Most of the results in compressed sensing are based on a random measurement process, meaning that the entries of the measurement matrix $A\sg{\in \R^{m,N}}$ follow some random distribution. The most prominent distributions used in the literature are the Gaussian and the Rademacher distribution. We call $A\in \R^{m,N}$ \emph{Gaussian} if its entries are independently drawn from a renormalized normal
distribution, i.e.,
\begin{align}\label{eq:GaussianMatrix}
	A=m^{-1/2}\; [a_{ij}]_{ij=1}^{m,N} \quad \text{with} \quad  a_{ij} \sim \mathcal{N}(0,1).
\end{align}
On the other hand a \emph{Rademacher matrix} $A\in \R^{m,N}$ has its entries independently chosen to be $m^{-1/2}$ or $-m^{-1/2}$ with equal probability $1/2$, i.e.,
\begin{align}\label{eq:RademacherMatrix}
	A=m^{-1/2}\; [a_{ij}]_{ij=1}^{m,N} \quad \text{with} \quad  \mathbb{P}(a_{ij}=1)= \mathbb{P}(a_{ij}=-1)=1/2.
\end{align}	 

Those typically chosen matrices have the specific characteristic that they are \emph{centered}, i.e., the expected value of each entry is $0$. However, as it will turn out in this work, for the reconstruction of binary signals non-centered matrices have some advantages. This phenomenon was already observed in the recent publication \cite{JK} for the reconstruction of nonnegative-valued signals. Here, so-called \emph{$0/1$-Bernoulli matrices} have been used. In contrast to Rademacher matrices the entries are independently chosen to be either $0$ or $1$, i.e.,
\begin{align}\label{eq:01BernoulliMatrix}
	A=m^{-1/2}\; [a_{ij}]_{ij=1}^{m,N} \quad \text{with} \quad  \mathbb{P}(a_{ij}=1)= \mathbb{P}(a_{ij}=0)=1/2.
\end{align}
Note that \sg{these} matrices can be very easily constructed from Rademacher matrices $A\in \R^{m,N} $ by $\frac{1}{2}(\one + A)$, where $\one\in \R^{m,N}$ denotes the all-one matrix.}

\subsection{Reconstruction of Nonnegative-Valued Signals}
\sg{Binary signals are in particular nonnegative-valued, so that  results concerning recovery of such signals can be readily applied to binary ones. We will therefore dedicate this subsection to the reconstruction of nonnegative-valued signals.}

The task of reconstructing nonnegative-valued signals from few measurements has gained some interest over the last years. It has become evident that basis pursuit restricted to the positive orthant $\R^N_+:=\{x=(x_i)_{i=1}^N\in \R^N: x_i\ge 0, i=1,\dots,N\}$, i.e., the program
\begin{align}\label{P+}
 \min \|x\|_1 \quad \text{subject to}\quad Ax=Ax_0 \quad \text{and} \quad x\in \R_+^N,\tag{$P_+$}
\end{align}
has a strong performance at recovering nonnegative-valued, sparse signals. The relatively simple structure of the method allows it to be thoroughly analytically analyzed.

In \cite{Stojnic+}, Stojnic introduced a new nullspace property and derived precise bounds on the sufficient number of measurements needed to recover a given nonnegative-valued signal $x_0$ using \eqref{P+}. The random matrices were assumed to have null-spaces whose basis is distributed according to either the Gaussian or Rademacher distribution. 

In \cite{DT}, Donoho and Tanner presented a different, more geometric, analysis of the problem. Their argument relies on the fact that if $F$ is defined as the convex hull of the vectors $\set{e_i: i\in K}$, it holds: A certain nonnegative-valued signal $x_0$ having non-zero entries on a set $K$ is recovered by \eqref{P+} if and only if $AF$ is a $k$-face of the projected polytope $A\R^N_+$. They managed to compute the probability of such a face ''surviving'' the projection with a random \sk{matrix} $A$ having a distribution \sk{fulfilling} certain assumptions. Since we will use their result later on, let us repeat it here.

\begin{definition}
 Let $B$ be a random $(m \times N)$-matrix. We say that $B$
 \begin{enumerate}
  \item is \emph{orthant-symmetric} if for each diagonal matrix $S\in \R^{N,N}$ with diagonal in $\{-1,1\}^N$ and every measurable set $\Omega\subset\R^{m,N}$, it holds
  \begin{align}
   \prb{BS\in \Omega}=\prb{B\in \Omega}, 
  \end{align}
  \item is in \emph{general position} if every subset of $m$ columns is almost surely linearly independent,
  \item has \emph{exchangable columns} if for each permutation matrix $\Pi\in\R^{N,N}$ and every measurable set $\Omega\subset \R^{m,N}$
  \begin{align}
   \prb{B\in \Omega}=\prb{B\Pi \in \Omega}.
  \end{align}
 \end{enumerate}
We further say that a random subspace $V\subset \R^N$ with basis $\{a_1,\dots, a_{N-m}\}$ is a \emph{generic random subspace} if the matrix $A:=[a_1 \dots a_{N-m}]^T$ is in general position. We call $V$  \emph{orthant symmetric} if $A$ is.
\end{definition}

Having introduced these notions, we can recall a specific result \sg{(Lemmas 2.2, 2.3) from \cite{DT} \sk{which} we will \sg{also} need in the sequel. It states that with high probability, a $k$-face of the orthant $\R^N_+$ survives under a random projection. More precisely it states the following: Let $F$ denote a $k$-face of the orthant $\R^N_+$ and \sk{let} $A\in \R^{m,N}$, with $m<N$, be in general position and have an orthant symmetric nullspace. Then
 \begin{align}\label{eqn:survivingFaces}
  \prb{AF \text{ is a } k-\text{face of } A\R^N_+}=1-P_{N-m, N-k},
 \end{align}
where
\begin{align}
 P_{ij}=2^{-j+1}\sum_{l=0}^{i-1}{i-1\choose l}, \quad {i,j\in \N}.
 \end{align}
%
This result }in particular gives the probability of success of the program \eqref{P+} recovering a certain $k$-sparse nonnegative-valued vector $x_0 \in \R^N$.

\subsection{Reconstruction of Binary Signals}

Not only for the reconstruction of nonnegative-valued signals \sg{several} results are already known, but also for binary signals. In this case, the canonical approach is to use the following adaptation of basis pursuit, to which one typically refers to as \emph{basis pursuit with box constraints}:

\begin{align}\label{Pbin}\tag{$P_{\text{bin}}$}
	\min\|x\|_1 \quad \text{subject to} \quad Ax=b \quad \text{and}\quad x\in [0,1]^N.
\end{align}

In \cite{DTHC} this algorithm was considered for the recovery of \emph{$k$-simple signals}. Those are signals having entries in $[0,1]$ with at most $k$ entries not equal to either $0$ or $1$. Hence, binary signals fall into the class of $0$-simple signals for \emph{any} sparsity level. Consequently, the analysis of \cite{DTHC} did not take sparsity into account.

In \cite{Stojnic+}, however, performance guarantees for sparse binary signals have been proven. It was shown that at most $N/2$  measurements are needed to recover a binary signals via \eqref{Pbin}. For values $k\ge N/2$, approximately $N/2$ will be needed, whereas the number can be reduced for $k<N/2$ (see the blue/solid curve in Figure \ref{fig:PhaseTransMirr}). 

The following null space condition \sg{has been shown to be sufficient \cite{Stojnic+} and necessary \cite{KKLP} for the success of \eqref{Pbin}. The vector $\mathds{1}_K$ is the unique solution of \eqref{Pbin} with $b=A\mathds{1}_K$ if and only if
\begin{align}
 \ker{A}\cap N^+\cap H_K =\{0\} \label{BNSP} \tag{B-NSP},
\end{align}}
where $N^+=\{w\in \R^N:\sum_{i=1}^N w_i\le 0\}$ and $H_K=\{w\in \R^N: w_i\le 0\text{ for } i\in K, \text{ and }
w_i\ge 0 \text{ for } i\in K^C\}$.


To ensure robust recovery from noisy measurements $b= Ax_0+n$, with $\|n\|_2\le \eta$, the following adaptation of \eqref{Pbin} has been considered \sg{(e.g. in \cite{KKLP})}:
\begin{align}
\label{BPBinRob}
\min\|x\|_{1} \quad \text{subject to} \quad \|Ax - b \|_{2} \le \eta \quad \text{and} \quad x\in [0,1]^N \tag{P$_{\text{bin}}^{\eta}$}.
\end{align}
Note \sg{that in order to define this algorithm properly the noise level $\eta$ is required to be known in advance.} 

By applying the framework of \emph{statistical dimensions} \cite{ALMT}, the authors of \cite{KKLP} \sg{estimated} the number of Gaussian measurements $\Delta_{bin}(k)$ needed to recover a $k$-sparse binary signal with high probability using \eqref{Pbin}. A plot of $\Delta_{bin}$ as well as the results of a numerical experiment validating the bound, is shown in Figure \ref{fig:Pbin}. \sk{This experiment is specified in Subsection \ref{subsec:BiasedMatrices}.}

\begin{figure}[H]
	\centering
	\captionsetup{justification=centering}
	\subfigure[Reconstruction error of \eqref{Pbin}.]{\includegraphics[width=0.43\textwidth]{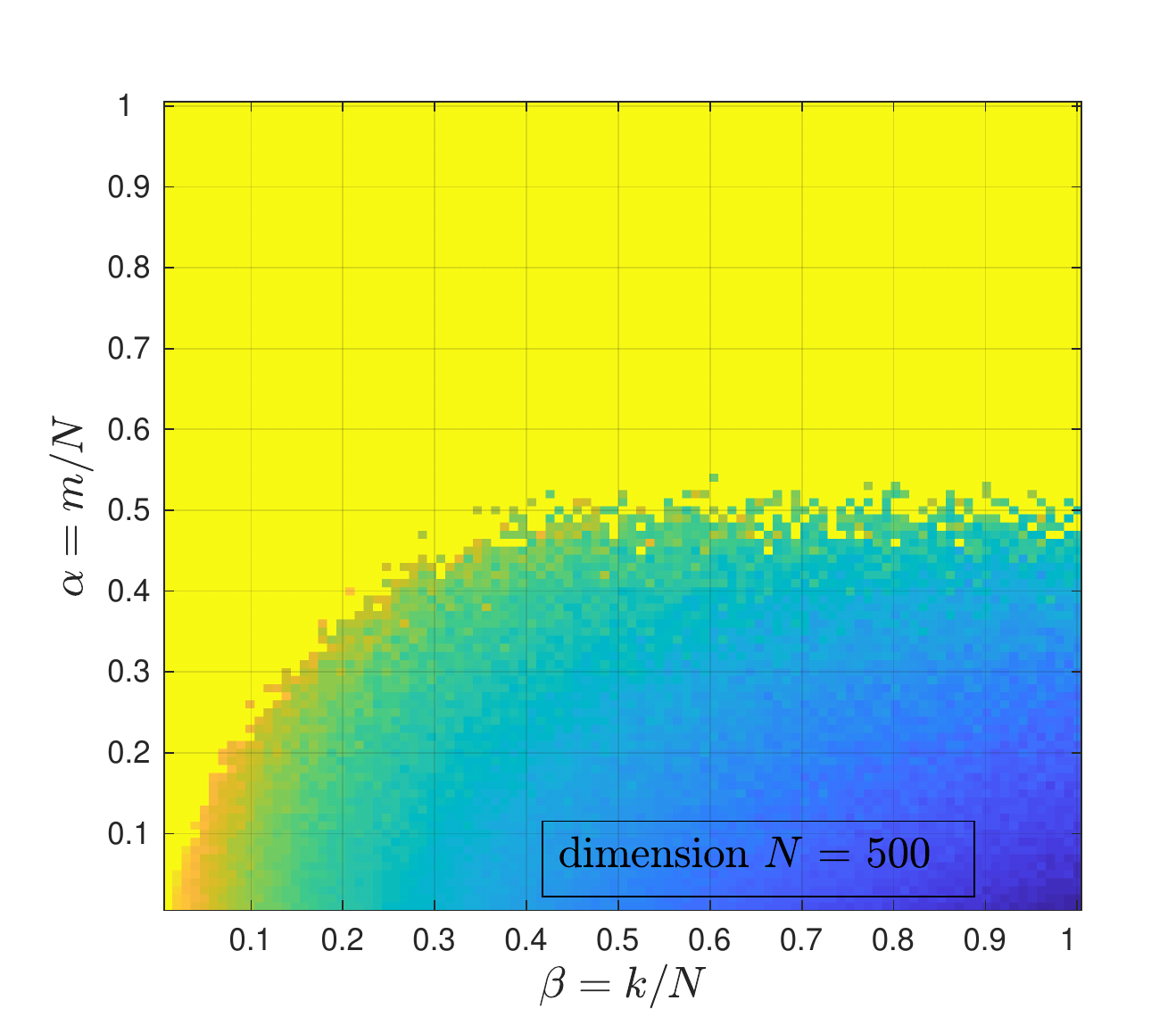}\label{fig:bink}}
	\hspace*{0.5cm}
	\subfigure[In \cite{Stojnic+, KKLP} theoretically proven phase transition of \eqref{Pbin} in comparison to the classical algorithm $(P_+)$ and $(P_1)$ for arbitrary $N\in\N$. Successful Reconstruction is guaranteed in the area above the curves with high probability.]{\includegraphics[width=0.47\textwidth]{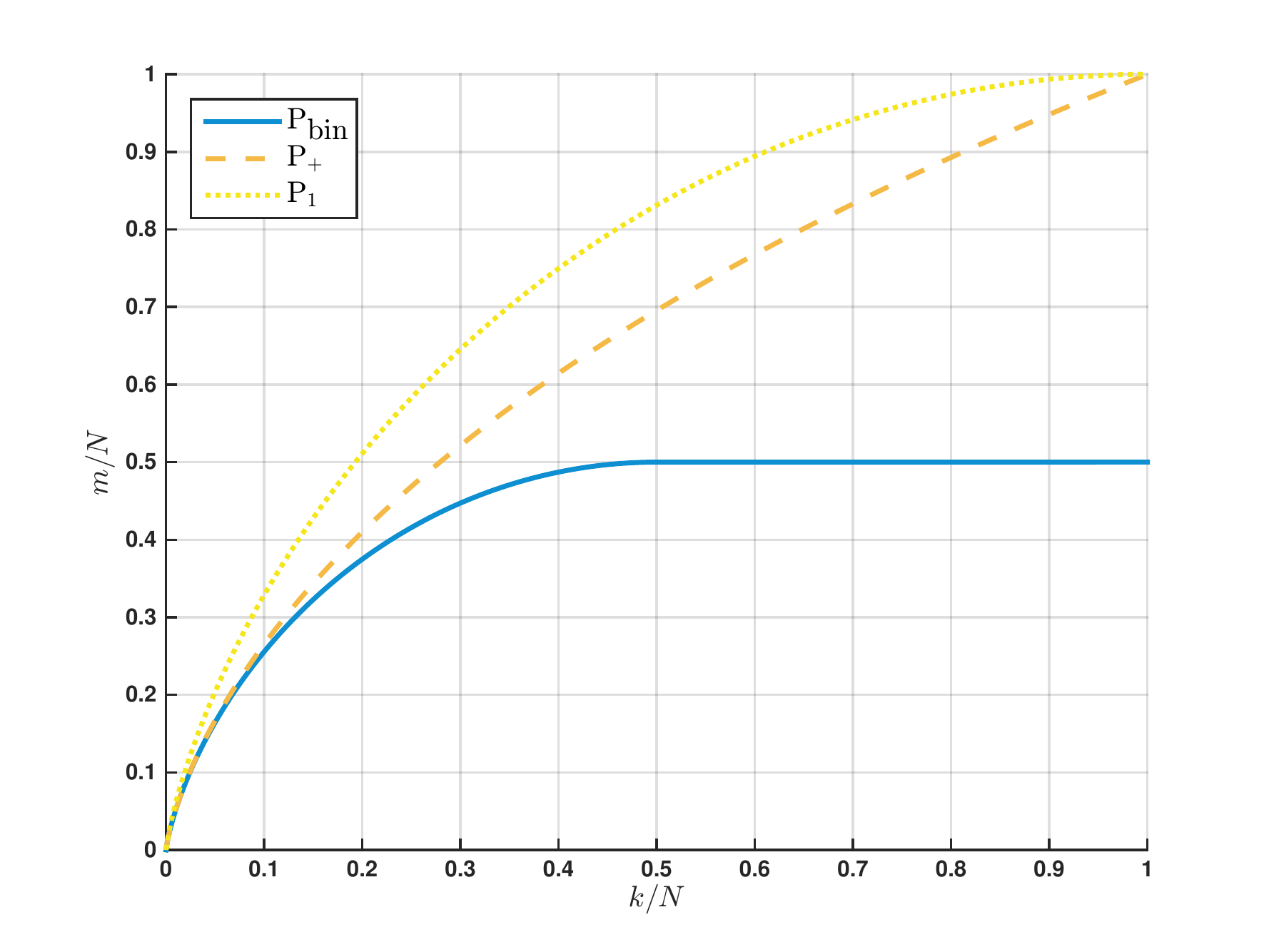}\label{fig:PhaseTrans}}
	\caption{Numerically and theoretically derived phase transition of \eqref{Pbin} \sk{from Gaussian measurements}.\\ \sk{\footnotesize{The experiment yielding this numerics is explained in more detail in Subsection \ref{subsec:BiasedMatrices}}}. \label{fig:Pbin}}
\end{figure}
 
\subsection{Main results}
Up until now, most measurement matrices \sk{that have been} considered \sk{were} centered, i.e.\sk{,} the expected value of each entry \sk{was assumed} to be $0$.  A simple numerical experiment reveals that\sk{,} when recovering sparse binary signals using \eqref{Pbin}, this might not be optimal. In Figure \ref{PbinBiased}, we have repeated the experiments used to generate Figure \ref{fig:Pbin}, but with $A$ being a $0/1$-Bernoulli matrix instead of a Gaussian. Two observations can be made:
\begin{enumerate}
\item For both the Gaussian and Bernoulli distribution, \sg{the numerical experiments indicate that using} $m > N/2$ measurements secures recovery with high probability, independent of the sparsity level.

\item In the Bernoulli case, and \sg{not in the Gaussian} case, \sg{the numerical experiments suggest that} the recovery of a sparse binary signal is equally probable to the recovery of an \sg{\emph{saturated}} binary signal, \sk{i.e., an signal which has only a few entries equal to zero.}

\end{enumerate}
\begin{figure}
	\centering 
	\captionsetup{justification=centering}
	\includegraphics[width=0.48\textwidth]{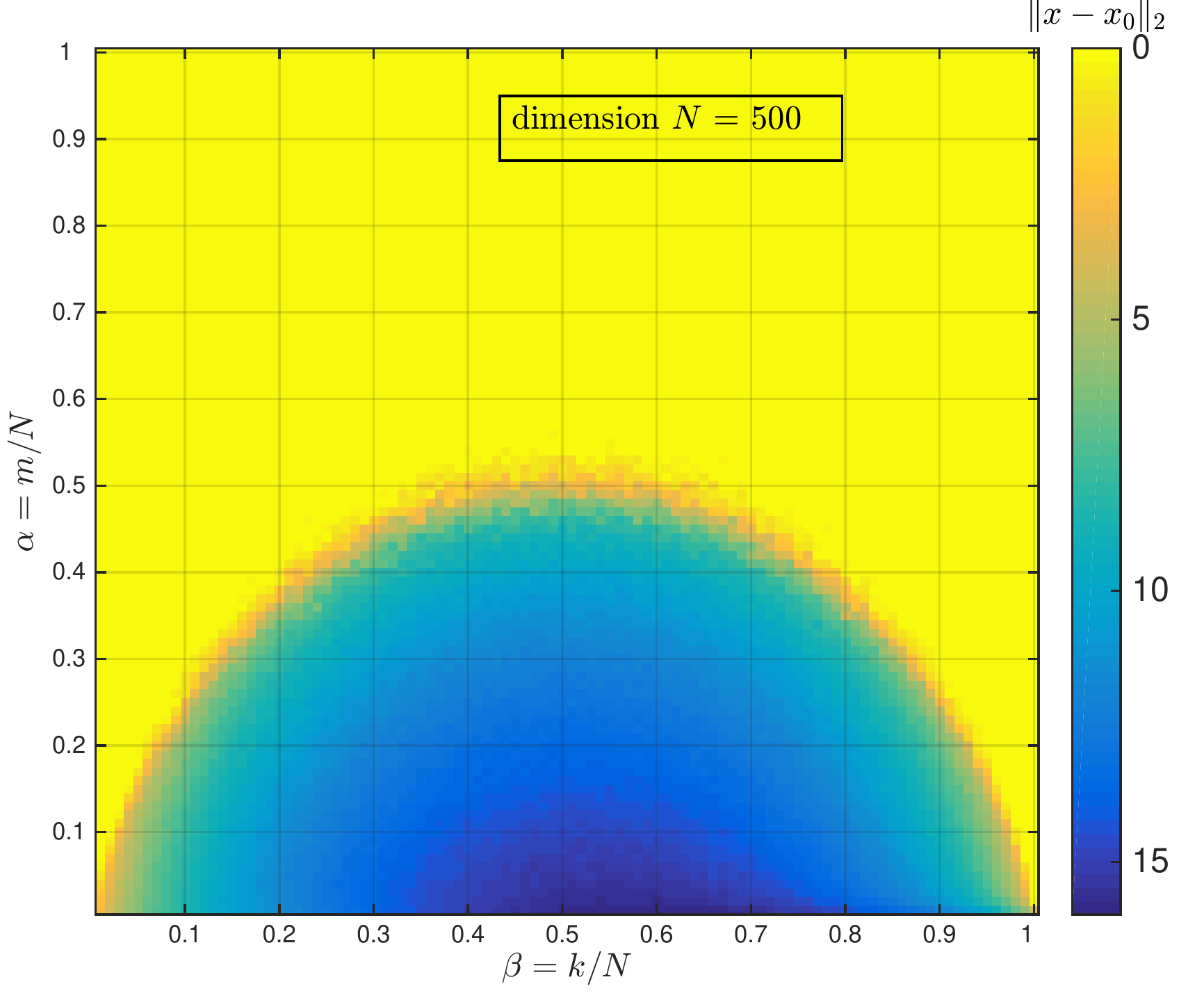}\label{fig:BiasedBP}
	\caption{Reconstruction from $0/1$-Bernoulli measurements \sk{via \eqref{Pbin}}. \\\sk{\footnotesize{The experiment yielding this figure is explained in more detail in Subsection \ref{subsec:BiasedMatrices}. \label{PbinBiased}}}}
	\end{figure}
We will provide statements \sk{which explain these observations} not only for Bernoulli matrices, but instead for \emph{biased} measurement matrices:
	\begin{align}
	 A = \mu \one + D. \label{BiasedMM}
	 \end{align}
	Here, $\mu\geq 0$ is a parameter (the expected values of the entries $a_{ij}$ of $A$), $\one \in \R^{m,N}$ is the matrix having only entries equal to $1$, and $D\sk{\in \R^{m,N}}$ is \sk{assumed to be} centered. We will make the simple assumption that the entries $d_{ij}$ of $D$ are i.i.d., with
	\begin{align} \label{eq:assumption}
		\erw{d_{ij}}=0, \ \erw{d_{ij}^2}=\sigma^2, \ d_{ij} \in \sk{[-\Lambda, \Lambda]} \text{ almost surely}.
	\end{align}
	
\sk{To verify} the first of the above observations, we \sk{will prove} the following result:

\begin{theorem}[Simplified Version of Theorem \ref{thm:largerN/2}]\label{thm:simple1}
Let $x_0\in \{0,1\}^N$ be a binary vector, and $A\in \R^{m,N}$  be a random matrix of the form \eqref{BiasedMM}, with \sg{some} additional assumption\sg{s.}
 Then if $m> N/2$, 
$x_0$ will be the solution to \eqref{Pbin} with high probability.
\end{theorem}

\sk{Note that this theorem in particular holds true for $\mu=0$, so that $A$ does not need to be biased.} As for the second of the observations, we \sk{will show} the following result:
	
	\begin{theorem}[Simplified version of Theorem \ref{thm:phaseTrans}]\label{thm:simple2}
	Let $A\in \R^{m,N}$ be a biased measurement matrix as described in \sk{\eqref{BiasedMM} and} \eqref{eq:assumption} with $\mu>0$, and $x_0\in \{0,1\}^N$ a $k$-sparse binary vector. Under the assumption\sg{
\begin{align}\label{eqn:Simplified2}
m \gtrsim \abs{\min(k, N-k)}\log \left(N\right),
\end{align}}
$x_0$ will be the unique solution to \eqref{Pbin} with high probability. In fact, under the same assumption, $x_0$ can be recovered by  instead solving the problem	
\begin{align}
		\min \norm{Ax -b}_2 \text{ subject to } x \in [0,1]^N. \tag{$\mathcal{P}_{LS,bin}$} \label{LSbin}
	\end{align}

\sg{If $m \gtrsim \abs{\min(k, N-k)}\log \left(N\right)$, with a constant larger than that in \eqref{eqn:Simplified2},} the solution $x_*$ of \eqref{LSbin} for $b= Ax_0 + n$ with \sg{$n\in \R^m$ and} $\norm{n}_2 \leq \epsilon$ obeys
\begin{align*}
\norm{x_0-x_*}_2 \leq \sqrt{\frac{9\left(\frac{16\sigma^2}{\mu^2}+ \min(k, N-k)\right)}{m \sigma^2}}\cdot \epsilon.
\end{align*}
\end{theorem}

Note that Theorem \ref{thm:phaseTrans} indicates that \eqref{Pbin} will be successful with few measurements both when $x_0$ is sparse, and far from being sparse $(k \approx N)$. An intuitive reason why this could be the case is that if $x_0$ is far from being sparse and binary, $\one - x_0$ will be sparse and binary. Recovering $\one - x_0$ should hence require few measurements. The theorem indicates that the problem \eqref{Pbin}, when the measurement matrix is biased, somehow automatically decides which of the two vectors $x_0$ and $\one - x_0$ should be tried to be recovered. It also shows that the bias of the measurements is crucial --  as $\mu\to 0$, the bound turns into a trivial one.

The fact that \eqref{LSbin} can be used for recovery instead of \eqref{Pbin} could possibly have a practical impact, as the former program is less complex. Also note that in contrast to \eqref{BPBinRob}, the noise level does not need to be known to properly apply \eqref{LSbin}.
	
	\begin{remark} 
	Note that the important case of $A$ being a $0/1$-Bernoulli matrix satisfies the assumptions \eqref{eq:assumption}. \sk{The requirements of} \eqref{eq:assumption}, however, exclude unbounded distributions for $D$, such as the Gaussian distribution. Since, e.g., the Gaussian distributions enjoy concentration inequalities ($\abs{d_{ij}} \leq \Lambda$ with large probability), we can, however, use conditioning to derive statements also for such matrices, that is
	\begin{align*}
		\prb{ A \text{ enjoys property P } } \leq \prb{ \exists \ (i,j) : \abs{d_{ij}} > \Lambda } + \prb{ d_{ij} \in [-\Lambda, \Lambda] \text{ and } A \text{ enjoys property P}}.
	\end{align*}
	\end{remark}
	
	The rest of the paper will be arranged as follows. In Section \ref{sec:Symmetry}, we will present both the open problems described in this section in more detail as well as the idea of the proof of the main result. In Section \ref{sec:Main}, we then present the details concerning the main results.
	
	 \sg{To complete the introduction,} let us present notation which will be used throughout the whole paper. The support of a signal will usually be denoted by $K\subset [N]=\{1,\dots,N\}$, a biased measurement matrix will be denoted by $A\in \R^{m,N}$ and a non-biased measurement matrix by $D\in \R^{m,N}$. \sk{The notation} $\sigma^2=\sigma^2(d)$ will stand for the variance of some random variable $d$. Furthermore, $\sprod{M}$ will denote the the linear hull of a set $M \sse \R^N$, and $\Pi_V$ the orthogonal projection onto the subspace $V$. 
\section{Utilizing the Symmetry of the Ground-Truth Signal}

As has been hinted in the introduction, signals $x_0\in\R^N$ having entries in the alphabet $\{0,1\}$, \sk{enjoy} the special property that also $\one-x_0$ is binary. In the following the goal is to utilize this additional structural property. Before treating the biased measurement matrices and the main results of this paper, let us briefly explore another means of exploiting the symmetry of binary signals. \sg{The following algorithm} was proposed by one of the authors of this article, together with co-authors, in \cite{KKLP}.

\subsection{Mirrored Binary Basis Pursuit with Box Constraints}
In \cite{KKLP}, the following observation was made: In case we knew in advance that $\|\one-x_0\|_0<\|x_0\|_0$, we could run the algorithm
\begin{align}\label{PbinS}\tag{$P_{\text{Mbin}}$}
 \min\|\one-x\|_1 \quad \text{subject to} \quad Ax=b \quad \text{and}\quad x\in [0,1]^N
\end{align}
to recover $x_0$. The authors of \cite{KKLP} proposed to combine \eqref{Pbin} and \eqref{PbinS} to form a new recovery algorithm, \emph{Mirrored Binary Basis Pursuit} (Algorithm \ref{alg:MiBiBP}).

\begin{algorithm}
 \DontPrintSemicolon
 \begin{enumerate}
 \item Solve $\min\|x\|_1 \text{ subject to } Ax=b \text{ and } x\in [0,1]^N$ to obtain a solution $\hat{x}_1$.
 \item Compute $\min\|\one-x\|_1 \text{ subject to } Ax=b \text{ and } \quad x\in [0,1]^N$ to obtain a solution $\hat{x}_2$.
 \item Let $\hat{x}=\argmin_{i=1,2}{\round{\hat{x}_i}-x_i}$, where $\round{x}$ denotes the vector in $\mathbb{Z}^N$ which is closest to $x\in \R^N$.
\end{enumerate}
\caption{Mirrored Binary Basis Pursuit (MiBi-BP)\label{alg:MiBiBP}}
\end{algorithm}

The idea of the algorithm is to first obtain a solution $\hat{x}_1$ of $\eqref{Pbin}$, which would be a good solution if $x_0$ is sparse and \sg{further} to solve \eqref{PbinS} to obtain a solution $\hat{x}_2$, which would be close to $x_0$, if $\one-x_0$ is sparse. The algorithm then chooses the solution which is closest to be\sg{ing} binary-valued. In \cite{KKLP} it has also been shown, that one of the solutions is indeed exactly binary, provided we have \sg{sufficiently many} measurements. However, this algorithm can only succeed in the case that only one of the solutions $\hat{x}_1$ and $\hat{x}_2$ is binary. In \cite{KKLP} the proposed algorithm has been validated numerically (see Figure \ref{fig:MirorredBPNum}). However, a theoretical validation remained open. This theoretical validation is provided by the following theorem:

\begin{figure}
	\centering
	\captionsetup{justification=centering}
	\subfigure[Reconstruction error of MiBi-BP. The experiment is conducted as the experiments for Figure \ref{fig:Pbin} with MiBiBP as reconstruction algorithm.]{\includegraphics[width=0.43\textwidth]{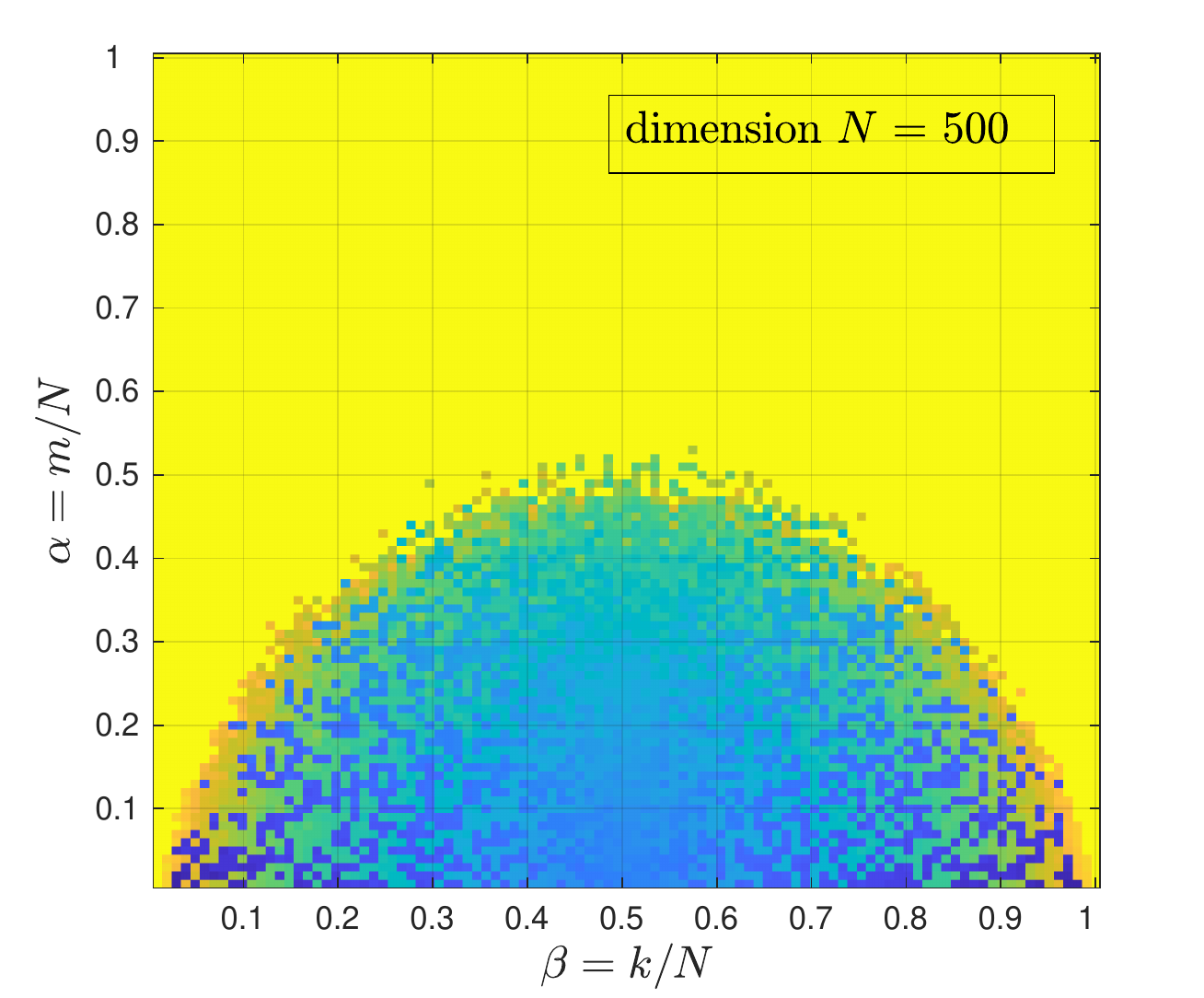}\label{fig:MirorredBPNum}}
	\hspace*{0.5cm}
	\subfigure[Phase transition of MiBi-BP for arbitrary \sk{$N\in\N$. Successful Reconstruction is guaranteed in the area above the curves with high probability.}]{\includegraphics[width=0.47\textwidth]{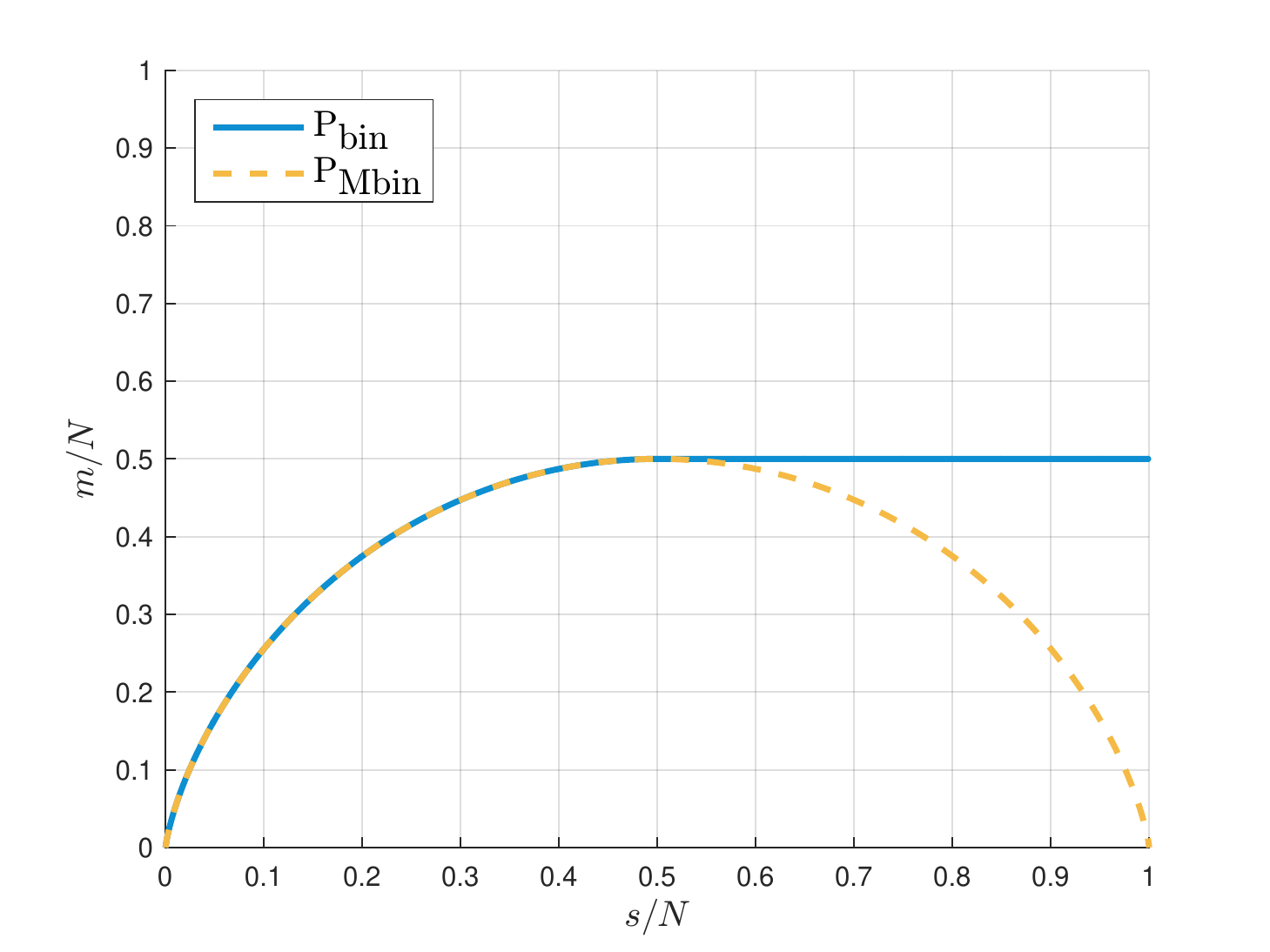}\label{fig:PhaseTransMirr}}
	\caption{Numerically and theoretically derived phase transition of the algorithm MiBiBP.}\label{fig:MirroredBP}
\end{figure}

\begin{theorem}
 Let $A\in \R^{m,N}$ be a Gaussian matrix and $x_0=\one_K$, with $K\subset [N]$ and $|K|\le k$. Further let $\varepsilon>0$ be some fixed tolerance. Provided $$m\ge \min\{\Delta_{\text{bin}}(k),\Delta_{\text{bin}}(N-k)\}+\sqrt{8\log(4/\varepsilon)N},$$
 \sg{where
 \begin{align}\label{eqn:Deltak}
 \Delta_{\text{bin}}(k):=\inf_{\tau\ge0}\left\{k\int_{-\infty}^{\tau}(u-\tau)^2\phi(u)du + (N-k)\int_{\tau}^{\infty}(u-\tau)^2\phi(u)du\right\}
\end{align}
and $\phi(u)\!=\!(2\pi)^{-1/2}e^{-u^2/2}$ being the probability density of the Gaussian distribution,} $x_0$ is the unique solution of MiBi-BP with probability larger than $1-\varepsilon$. 
\end{theorem}

 In Figure \ref{fig:Pbin} we provide a plot of the function $\Delta_{\text{bin}}$.

\begin{proof} 
 Without loss of generality, let $k\ge N/2$. Since $m \ge \Delta_{\text{bin}}(n-k)$, Theorem 2.7 of \cite{KKLP} implies that with probability larger $1-\varepsilon$, $x_0$ is the unique solution of \eqref{PbinS}.
 Towards a contradiction, assume that $x_0$ is not the unique solution of MiBi-BP \sg{and} let $\hat{x}\neq x_0$ be the solution of MiBi-BP. Then it needs to hold that $\hat{x}$ is a solution of \ref{Pbin} and $\|\rnd(\hat{x})-\hat{x}\|\le \|\rnd(x_0)-x_0\|=0$, \sg{implying} that $\hat{x}$ is binary. Thus, $x_0-\hat{x}\in \ker(A)\cap (\mathcal{A}-\mathcal{A})$. In \cite{FK} it was shown, that for matrices $A$ with i.i.d. columns from a non-singular distribution, we have $\ker{A}\cap \mathcal{A}-\mathcal{A}=\{0\}$ with probability $1$. This yields $\hat{x}=x_0$, since the Gaussian distribution is non-singular. 
\end{proof}

The value for $\min\{\Delta_{\text{bin}}(k),\Delta_{\text{bin}}(N-k)\}$ and therefore the phase transition for MiBi-BP is illustrated in Figure \ref{fig:PhaseTransMirr}.

The main disadvantage of MiBi-BP is that it, of course, has twice the runtime of standard basis pursuit. In the remainder of the paper, we will be devoted to proving that when using biased measurement matrices, the mirroring procedure is unnecessary.

\subsection{Using Biased Measurement Matrices}\label{subsec:BiasedMatrices}
Let us describe the numerical experiment leading to Figures \ref{fig:Pbin} and \ref{PbinBiased} in more detail.  The ambient dimension $N$ is chosen to be $N=500$. For each combination of sparsity level $k\in \{0.01,0.02, \dots, 1\}\cdot N$ and
number of measurements $m\in \{0.01,0.02, \dots, 1\}\cdot N$, we first draw a Gaussian \sk{$D\in \R^{m,N}$ (for Figure \ref{fig:Pbin}) and Rademacher matrix $\tilde{D}\in \R^{m,N}$} and set $A=\one+\tilde{D}$ (for Figure \ref{PbinBiased}). We further choose some random permutation of the numbers $1,\dots, N$ and set the entries of the vectors $x_0$ which correspond to the first $k$ entries of the permutation to one and all others to zero. We then solve \eqref{Pbin} \sk{for $b=Dx_0$ (Figure \ref{fig:Pbin}) and} for $b=Ax_0$ \sk{(Figure \ref{PbinBiased})}. We repeat the procedure for each combination of sparsity level and number of measurements $25$ times.

To obtain an idea why the phenomenon \sk{we can observe in Figure \ref{PbinBiased}} appears, we also recorded the cases in which either \emph{both} $\one_K$ and $\one-\one_K$ or \emph{none of the two} were recovered by \eqref{Pbin}. As can be seen in Figure \ref{fig:Simultaneous}, this was almost always the case, the only exception being the phase transition region in Figure \ref{PbinBiased}.  For Gaussian matrices, a corresponding experiment reveals that this is true with high probability instead only in the case when the number of measurements $m$ exceeds $N/2$.

These numerical observations suggest that we should try to investigate when a matrix has the property that both $\one_K$ and $\one-\one_K$ is recovered by \eqref{Pbin}. Concerning the simultaneous recovery of different binary sparse signal, the following theorem \sg{holds}. 
%

\begin{figure}[H]
	\centering
\includegraphics[width=0.48\textwidth]{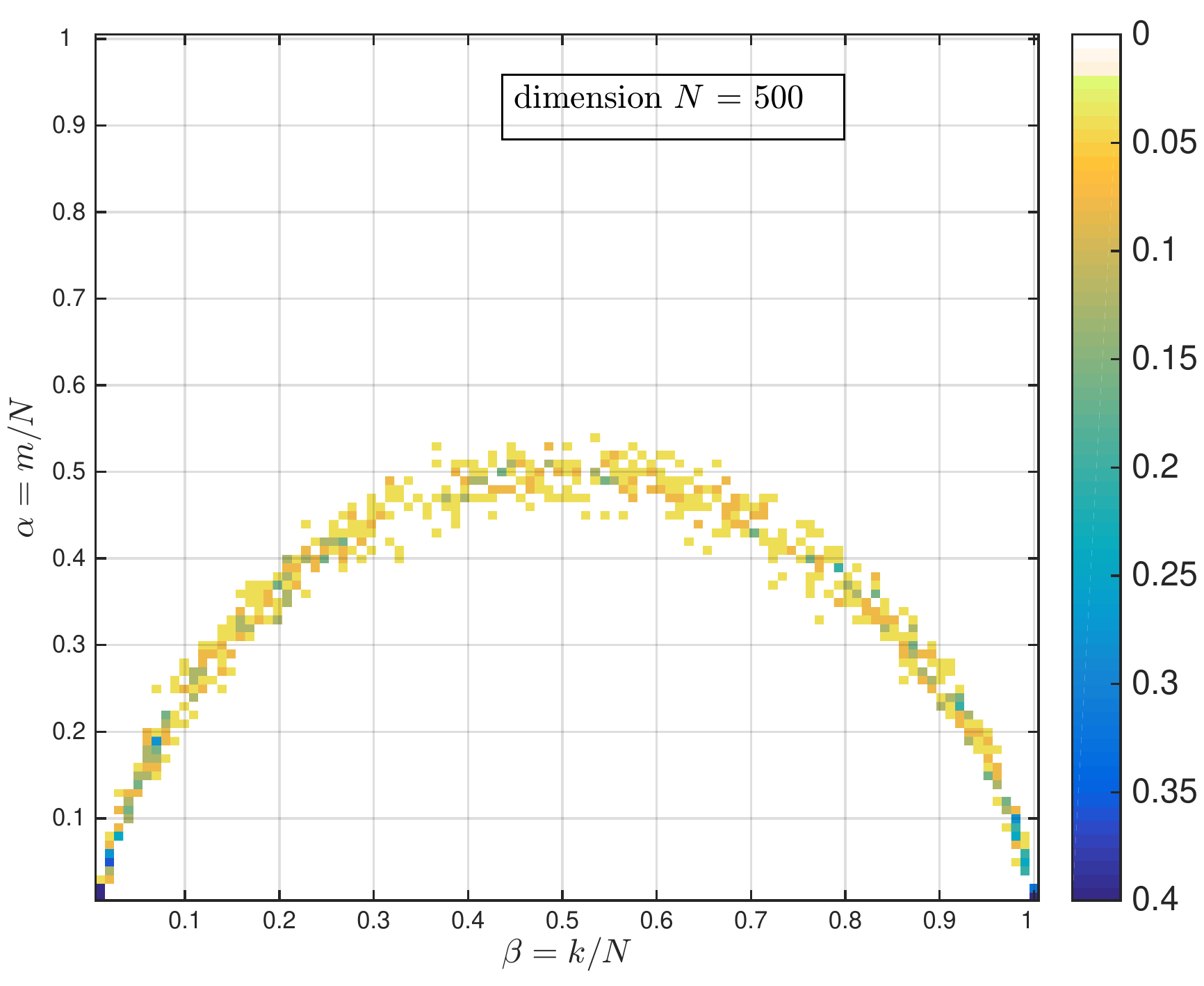}
	\caption{The probability for simultaneous reconstruction, or simultaneuous non-reconstruction, of $\one_K$ and $\one - \one_K$, for biased measurement matrices. \label{fig:Simultaneous}}
\end{figure}

\begin{theorem}\label{thm:newNSP}
 Let $A\in \R^{m,N}$ and $K,S\subset [N]$. Suppose $\one_K$ is the unique solution of \eqref{Pbin} with $b=A\one_K$ and $\one_S$ with $b=A\one_S$. Then 
 \begin{align}
  \ker(A)\cap H_{K} \cap H_{S^c}=\{0\}.
 \end{align}
\end{theorem}

\begin{proof}
 Let $w$ be a nonzero element in $\ker(A)\cap H_K \cap H_{S^c}$. Then in particular $w\in \ker(A)\cap H_K$. Since $\one_K$ is the unique solution of \eqref{Pbin}, \sg{\eqref{BNSP} with respect to $K$ holds, and thus, $ \sum_{i\in[N]}w_i>0$}. Similarly,  $w\in \ker(A)\cap H_{S^c}$  implies $-w\in \ker(A)\cap H_S$ and consequently, $-\sum_{i\in [N]}w_i>0$ due to B-NSP with respect to $S$. This is a contradiction.
\end{proof}

For $S=K^c$, it turns out that the implication in the previous proposition is in fact an equivalence, and that both conditions are equivalent to
$\{x\in [0,1]^N:Ax=Ax_0\}$, or $\{x\in [0,1]^N:Ax=A((\one-x_0)\}$, being a singleton. The following proposition \sg{makes this statement precise.}

\begin{proposition}\label{prop1} Let $A\in \R^{m,N}$ and \sg{$K\subset [N]$}. Then the following \sg{conditions} are equivalent.
 \begin{enumerate}
  \item[(i)] $\one_K$ and $\one-\one_K=\one_{K^c}$ are the unique solutions of \eqref{Pbin} with $b=A\one_K$ and $b=A\one_{K^c}$, respectively.
  \item[(ii)] $\ker(A)\cap H_K=\{0\}$.
  \item[(iii)] $\{x\in [0,1]^N:Ax=A\one_K\}=\{\one_K\}$.
  \item[(iv)] $\{x\in [0,1]^N:Ax=A\one_{K^c}\}=\{\one_{K^c}\}$.
 \end{enumerate}
\end{proposition}

\begin{proof} $(i) \Rightarrow (ii)$ is a special case of Theorem \ref{thm:newNSP}.\sk{ $(ii)\Rightarrow (i)$ is clear, since  $\ker(A)\cap H_K=\{0\}$ by symmetry implies $\ker(A)\cap H_{K^c}=\{0\}$ and  $\ker(A)\cap H_{K}\cap N^+=\{0\}$.
	
It remains to prove that $(ii)$ is equivalent to $(iii)$ and $(iv)$.

Let us concentrate on $(ii) \Leftrightarrow (iii)$, since $(ii) \Leftrightarrow (iv)$ is similar. \sg{For this,} let $v\in  \ker(A)\cap H_{K}$ and assume w.l.o.g. $\|v\|_{\infty}\le 1$. Then $x_0+v\in [0,1]^N$ and $A(x_0+v)=Ax_0$, i.e., $x_0+v\in \{x\in [0,1]^N:Ax=Ax_0\}$. Therefore, $(iii)$ implies $v=0$.

 For the other direction} assume $\ker A \cap H_K = \set{0}$ and, towards a contradiction, that there exists \sg{some} $w \in \set{x\in [0,1]^N:Ax=Ax_0}$ not equal to $x_0$. \sg{This implies that $\eta =w -x_0$} lies in the kernel of $A$, and also 
\begin{align*}
	\eta(i) &= w(i) - x_0(i) =  w(i) -1 \leq 0, \quad i \in K \\
	\eta(i) &= w(i) - x_0(i) =  w(i) -0 \geq 0, \quad i \notin K ,
\end{align*} 
 since $w \in [0,1]^N$. Hence, $0 \neq \eta \in \ker A \cap H_K$, which is a contradiction.
 \end{proof}
Note that, due to $(iii)$ of Proposition \ref{prop1}, the property $\ker(A)\cap H_K=\{0\}$ in particular \sg{implies} that there \sg{does not exist another} solution of $Ax=b$ in the box $[0,1]^N$. \sg{Hence,} it is redundant to look for the solution with the smallest $\ell_1$--norm. Thus, if $\ker(A)\cap H_K=\{0\}$ is true,  we could run $\ell_2$-minimization instead of basis pursuit, i.e., the program 
\begin{align}\label{LSbin}
	\min \norm{Ax-b}_2 \quad \text{subject to} \quad x\in [0,1]^N. \tag{$\calP_{LS, bin}$}
\end{align}
 \sg{Numerically this can indeed be observed} as illustrated in Figure \ref{fig:BoxLS}. Note, \sg{that \eqref{LSbin}} has important advantages compared to basis pursuit, in particularly regarding complexity and noisy measurements. \sk{Note that in order to draw Figure \ref{fig:BoxLS} we designed the experiment similar to those before (cf. Figure \ref{fig:MirroredBP}). Hence, for each combination of sparsity level $k\in \{0.01,0.02, \dots, 1\}\cdot N$ and
number of measurements $m\in \{0.01,0.02, \dots, 1\}\cdot N$, $N=500$, we draw a Rademacher matrix $\tilde{D}\in \R^{m,N}$ and set $A=\one+\tilde{D}$. We further choose some random permutation of the numbers $1,\dots, N$ and set the entries of the vectors $x_0$ which correspond to the first $k$ entries of the permutation to one and all others to zero. We then solve \eqref{LSbin} for $b=Ax_0$. We repeat the procedure for each combination of sparsity level and number of measurements $25$ times.}  \newline

 \begin{figure}[H]
	\centering
	\captionsetup{justification=centering}
	\includegraphics[width=0.48\textwidth]{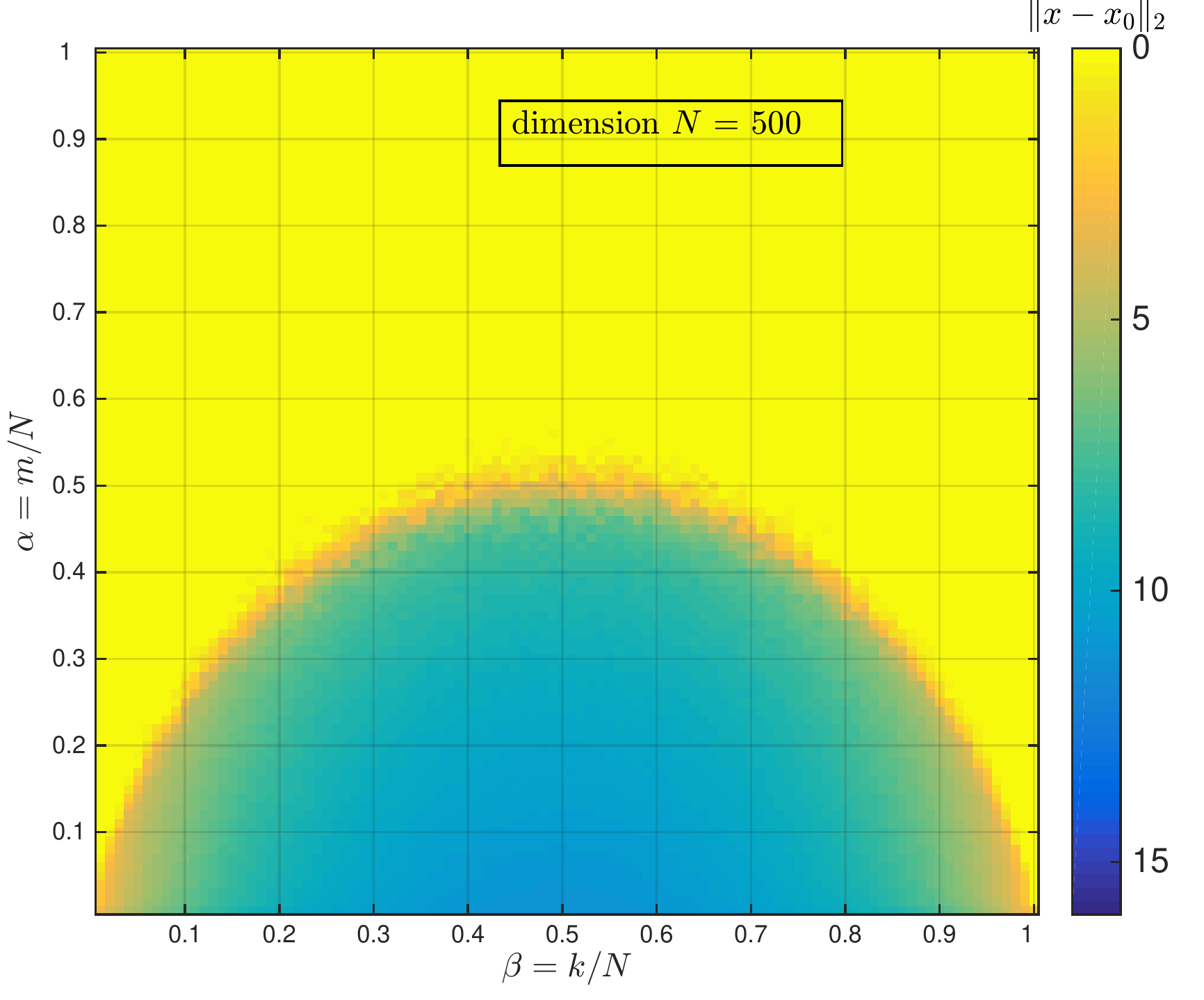}
	\begin{center}\caption{Reconstruction from $0/1$-Bernoulli measurements via \eqref{LSbin}.\\ \footnotesize{The experiment yielding this figure is explained in \sg{the paragraph after Equation \eqref{LSbin}.}} \label{fig:BoxLS}}\end{center}
\end{figure}
 
 By now, the strategy \sg{to prove our main results} should be clear: We should provide theoretical guarantees for $\ker(A)\cap H_K=\{0\}$ in the case that $A$ is a biased matrix of the form \eqref{BiasedMM}. These guarantees will be provided in the next section, in the form of the main results (Theorems \ref{thm:largerN/2} and \ref{thm:phaseTrans}).

 \label{sec:Symmetry}
\section{Main Results}

In this section, we will prove our main results.  We begin by deriving an equivalence between the condition $\ker(A) \cap H_K = \{0\}$ and one which is easier to resolve analytically.

\begin{proposition}\label{prop:HKPlus}
Let $A$ be an (arbitrary) element of $\R^{m,N}$ and $K\subset [N]$. Then the following statements are equivalent:
\begin{enumerate}
 \item[(i)] $\ker(A)\cap H_K=\{0\}$.
 \item[(ii)] $0\notin \mathcal{C}=\chull\left\{\{a_i:i\in K\},\{-a_i:i\notin K\}\right\}$
 \item[(iii)] $ \exists v \in \R^m \text{ such that } A^Tv \in H_K^+$,
 where $H_K^+=\{w\in \R^N: w_i<0 \text{ for } i\in K \text{ and } w_i>0 \text{ for } i\notin K\}$.
\end{enumerate}
\end{proposition}

Note that the difference between $H_K^+$ and $H_K$ is that we require strict inequalities for $H_K^+$. 

\begin{proof}
We start with proving the implication $(i)\Rightarrow (ii):$ Suppose $0\in \chull\left\{\{a_i:i\in K\},\{-a_i:i\notin K \}\right\}$.  Then there exists a $\lambda\in \R^N_+$ with $0=\sum_{i\in K}\lambda_i a_i-\sum_{i\notin K}\lambda_i a_i$. Thus, the vector $w\in \R^N$  defined through $w_i:=-\lambda_i$ for $i\in K$ and $w_i:=\lambda_i$ for $i\notin K$ fulfills $w\in \ker(A)\cap H_K$. This is a contradiction to the assumption that $\ker(A)\cap H_K$ is trivial.

To prove the implication $(ii)\Rightarrow (iii)$, suppose that $0\notin \mathcal{C}$. Then there exists a separating hyperplane $H_v=\{x\in \R^m:\langle v,x\rangle+c=0\}$, $v\in \R^m$, that strictly separates $0$ and $\mathcal{C}$, say $\langle v,x\rangle+c<0$ for $x\in \mathcal{C}$ and $\langle v,0\rangle+c>0$. Hence, $c>0$ and $\langle v,x\rangle<-c<0$ for $x\in \mathcal{C}$. Thus, due to the definition of $\calC$, \sg{we conclude that} $\langle v,a_i \rangle<0$ for $i\in K$ and  $\langle v,a_i \rangle>0$ for $i\notin K$, which means that $A^Tv\in H_K^+$.

It remains to  prove that $(iii)$ implies $(i)$. To see this, assume that $v \in \R^m$ with $A^Tv \in H_K^+$ exists, but $\ker(A)\cap H_K\neq\{0\}$. Then we find $0\neq w\in \R^N$ with  $w\in \ker(A)\cap H_K$. This implies $$0=v^TAw=(A^Tv)^Tw=\langle A^Tv, w \rangle >0, $$ 
which is a contradiction. In the last step, we utilized that $A^Tw$ and $v$ have the same sign pattern and that $(A^Tw)_i\neq 0 $ for all $i\in [N]$ and $v_i\neq 0$ for at least one index $i\in[N]$. 
\end{proof}

We now move on to prove our main results.

\subsection{\sg{Precise Statement and Proof of Theorem \ref{thm:simple1}}}

The proof of the first of the main results, which applies when the number of measurements exceeds $N/2$, also for unbiased matrices, can by now be dealt with relatively directly.

\begin{theorem}\label{thm:largerN/2}
Let $x_0\in \{0,1\}^N$ be some binary vector, and $A\in \R^{m,N}$  be a random matrix of the form \eqref{BiasedMM}, with the additional assumption that  the probability distribution of $d_{ij}$, $i=1,\dots,m, j=1,\dots,N$, is symmetric. If $m> N/2$, the following holds.

\begin{itemize}
\item[(i)] If \sg{the distribution of $D$} has a density with respect to the Lebesgue measure on $\R^{m,N},$ \sg{the set} $\{x\in [0,1]^N:Ax=Ax_0\}$ is a singleton with a probability larger than $1-P_{N-m,N}$, where
$$P_{N-m,N} :=  2^{-N+1} \sum_{\ell=0}^{N-m-1} \binom{n-1}{\ell}.$$ 
\item[(ii)] If $D$ is a Rademacher matrix, \sg{the set}
$\{x\in [0,1]^N:Ax=Ax_0\}$ is a singleton  with probability larger than $(1-P_{N-m,N})(1-2^{-m/2})$.
\end{itemize}
In particular, in both cases,  both $x_0$ and $\one-x_0$ will be succesfully be recovered by \eqref{Pbin}.
\end{theorem}
Note, that the case of a (possibly shifted) Gaussian is covered by $(i)$ and the case that $A$ is a $0/1$-Bernoulli matrix by $(ii)$.
\begin{proof}
\sg{
In the following we will argue that with high probability, there will exist a $w \in \sprod{\one}^\perp$
with $\sprod{w, d_i} <0$ for each $i \in K$ and $\sprod{w, d_i} >0$ for $i\notin K$. Since $\sprod{w,d_i} = \sprod{w,a_i}$ for each $i$ for $w \in \sprod{\one}^\perp$, assumption 
$(iii)$ of Proposition \ref{prop:HKPlus} would then be fulfilled and therefore we would have $\ker(A)\cap H_K=\{0\}$.

To this end we define the matrix $\widetilde{D}$, which we will interpret as a linear map from $\R^N$ to $\sprod{\one}^\perp$, formed by concatinating the vectors $(\Pi_{\sprod{\one}^{\perp}}(-d_i))_{i\in K}=: ( \tilde{d}_i)_{i\in K}$ and $(\Pi_{\sprod{\one}^{\perp}}d_i)_{i\notin K}=: ( \tilde{d}_i)_{i\notin K}$. This matrix is ortho-symmetric and has exchangable columns. These properties are inherited from the same properties of the matrix $D$. The latter further follows from the independence and symmetry of the $d_{ij}$ (see \cite{DTHC}, in particular pages 4 and 8).
 
 Now, as long as $\widetilde{D}$ is in general position, we can apply \sg{results from \cite{DT} (cf. \eqref{eqn:survivingFaces})} to conclude 
 \begin{align}
  \prb{0 \text{ is a vertex of } \cone(\tilde{d}_i)_{i\in [N]} }=\prb{\widetilde{D}0 \text{ is a vertex of } \widetilde{D}\R^N_+}=1-P_{N-m,N}.
 \end{align}  
Now $\widetilde{D}\R^N_+$ is a different notation for $\cone{(\tilde{d}_i:i\in [N])}$, and we thus have with probability $1-P_{N-m,N}$ that zero is a vertex of $\cone{(\tilde{d}_i:i\in [N])}$. However, if  $0$ is a vertex of $\cone{(\tilde{d}_i:i\in [N])}$ then we can either conclude $0\notin \chull{(\tilde{d}_i:i\in [N])}$ or that there is $i\in [N]$ with $\tilde{d}_i=0$. To see that this is true, suppose towards a contradiction that $0\in \chull{(\tilde{d}_i:i\in [N])}$ and $\tilde{d}_i\neq 0$, $i\in [N]$. This means that there are $\lambda_i>0$, $i\in [N]$, and $\sum_{i\in [N]}\lambda_i=1$ such that $0=\sum_{i\in [N]}\lambda_i\tilde{d}_i$. This in turn means that at least two of the coefficients, say with out loss of generality $\lambda_1$ and $\lambda_N$, are non-zero, because otherwise one of the $\tilde{d}_i$ must be zero. Thus we have $0=\sum_{i=1}^{N-1}\lambda_i\tilde{d}_i+\lambda_N\tilde{d}_N$, meaning that zero is contained in the line segment between   $\sum_{i=1}^{N-1}\lambda_i\tilde{d}_i$ and $\lambda_N\tilde{d}_N$, which contradicts zero being a vertex. Now, if $\widetilde{D}$ is in general position, non of the $\tilde{d}_i$ can be equal to zero and therefore we have $0\notin \chull{(\tilde{d}_i:i\in [N])}$. 

But if  $0\notin \chull{(\tilde{d}_i:i\in [N])}$, there is a separating hyperplane that strictly separates $0$ and  $\chull{(\tilde{d}_i:i\in [N])}$, i.e., there is $v\in \R^m$ such that $\langle v,\tilde{d}_i\rangle>0$. Hence we have
\begin{align}
0<\langle v,\tilde{d}_i\rangle=\pm \langle v,\Pi_{\sprod{\one}^{\perp}}d_i\rangle=\pm \langle\Pi_{\sprod{\one}^{\perp}} v,d_i\rangle,
\end{align}
thus with $w:=\Pi_{\sprod{\one}^{\perp}} v$ the claim is proven.}
%
%
%
%

It remains to argue that $\tilde{D}$ is in general position. If the
distribution of $D$ has a density with respect to the Lebesgue measure on $\R^N$, 
 $\Pi_{\sprod{\one}^\perp}D$ will also have, and therefore almost surely
 be in general position. This concludes the proof for the case $(1)$.

For the second case, let us note that if for each subset $L$ of indices
 with $\abs{L}=m-1$, \begin{align}
\left\{d_i:i\in L\right\} \cup \{\one\}
 \end{align}
 is a linearly independent system, the system 
 $\left\{\tilde{d}_i:i\in L\right\}$ will also be linearly independent.
 Applying Corollary 1.2 of \cite{BVW}, we however see that
 \begin{align}
 \prb{\left\{d_i:i\in L\right\} \cup \{\one\} \text{ is linearly
 dependent}}\lesssim 2^{-m/2}. \end{align}
 Thus with probability larger than $1-2^{-m/2}$, the columns of
 $\tilde{D}$
 are in general position also in the second case. This concludes the
 proof also for this case.
 
\end{proof}

\begin{remark}
Note that
\begin{align}
 P_{N-(N/2+1),N}=2^{-N+1}\sum_{l=0}^{N/2-1}{N-1 \choose l}\le 2^{-N+1}\sum_{l=0}^{(N-1)/2}{N-1 \choose l}\le 2^{-N+1}2^{-1}2^{(N-1)}=1/2,
\end{align}
and furthermore $P_{N-m,N} \sim 0$ for $m\ll N/2$. Thus, the probabilities described in the previous theorem are larger than $1/2$ (times $ (1-2^{-m})$ in the Rademacher case) for all values of $m>N/2$, and  very close to $1$ for $m \gg N/2$.
\end{remark}

\subsection{Precise Statement and Proof of Theorem \ref{thm:simple2}}
The proof of the second main result is slightly more involved that the one of the first. Let us start by deriving a dual certificate condition which will imply both $\ker A \cap H_K=\{0\}$ \sg{and} stability of the boxed-constrained least squares problem.

\begin{proposition} \label{eq:StabCert}
\sg{Let $s,t, \varepsilon>0$ and $A\in \R^{m,N}$}. Suppose that there exists a dual certificate $A^*\nu \in H_K^t$, where $H_K^t=\{w\in \R^N: w_i<-t \text{ for } i\in K \text{ and } w_i>t \text{ for } i\notin K\}$,
and additionally $\norm{\nu}_2 \leq s$.

Let $x_0=\mathbb{1}_K\in \R^N$ be the binary signal supported on $K$, and $b=Ax_0 + n$ with $\norm{n}_2 \leq \epsilon$. Then the solution $x_*$ of the program \eqref{LSbin} (or of the program \eqref{BPBinRob})
obeys
\begin{align*}
	\norm{x_*-x_0}_2 \leq \frac{2s}{t}\epsilon.
\end{align*}

\end{proposition}

\begin{proof}
	Let us first bound the \emph{$H_K$-restricted singular value}, which was introduced in \cite{AL}, of $A$. We have
	\begin{align*}
		\sigma_{H_K}(A) &= \min_{ x\in H_K , \norm{x}_2 =1} \norm{Ax}_2  = \min_{ x\in H_K , \norm{x}_2 =1} \sup_{\norm{p}_2 \leq 1} \sprod{Ax,p} \geq \min_{ x\in H_K , \norm{x}_2 =1} \sprod{Ax,\norm{\nu}_2^{-1} \nu}\\
		&\geq  \min_{ x\in H_K, \norm{x}_2 =1} s^{-1} \sum_{i=1}^{m} x_i (A^*\nu)_i  \geq \frac{t}{s} \min_{ x\in H_K, \norm{x}_2 =1} \norm{x}_1 \geq \frac{t}{s}>0.
	\end{align*}
	This readily implies that $\ker A \cap H_K = \set{0}$, so that, by Proposition \ref{prop1}, the solutions of \eqref{LSbin} and \eqref{BPBinRob} \sg{coincide}.
	
	Now we convert the bound on the restricted singular value into the error bound. Notice that due to $\norm{x_*}_\infty \leq 1$, $x_*\in [0,1]^N$ and the support assumption on $x_0$, we will have  $x_*-x_0 \in H_K$. Consequently,
	\begin{align*}
		\sigma_{H_K} (A) \norm{x_*-x_0}_2 \leq \norm{Ax_*-A x_0}_2 \leq \norm{Ax_* -b}_2 + \norm{Ax_0 -b}_2 \leq 2 \norm{Ax_0 -b}_2 \leq 2\epsilon,
	\end{align*}
	where we used the optimality of $x_{*}$  in the third step. We conclude
	\begin{align*}
		\norm{x_*-x_0}_2 \leq \frac{2\varepsilon}{\sigma(H_K)} \leq \frac{2s}{t}\epsilon.
	\end{align*}
\end{proof}

Let us now define the certificate we will work with in the following. For a sparse binary signal $\epsilon_0$ and a parameter $\rho \sg{=-\mu\sigma^2/4}$, \sg{where $\mu$ and $\sigma$ are specified by the measurement matrix $A$ (cf. Equations \eqref{BiasedMM} and \eqref{eq:assumption})}, we set
\begin{align} \label{eq:dualCertificate}
	\nu = \rho \one + D\epsilon_0 - m^{-1}\sprod{D\epsilon_0,\one} \one.
\end{align}

Let us begin by proving that, \sg{for a suitable value of $t>0$, we have} $A^* \nu \in H_{\supp \epsilon_0}^t$ with high probability.

\begin{lemma} \label{lem:certHk}
	Let $A\in \R^{m,N}$ be a biased measurement matrix of the form \eqref{BiasedMM} with $\erw{d_{ij}^2}=\sigma^2$ and $d_{ij}\in [-\Lambda, \Lambda]$, for $i\in [m], j\in [N]$. Further, let $\nu\in \R^m$ be defined as in \eqref{eq:dualCertificate} \sg{with $\rho= -\mu^{-1}\sigma^2/4$}.
	Under the assumption that $$m \gtrsim \max\left(\frac{\Lambda^2}{\mu^2}, \abs{J} \frac{\Lambda^4}{\sigma^4}\right)\log \left(\frac{N}{\varepsilon}\right),$$ \sg{we have} $A^* \nu \in H_{\supp \epsilon_0}^{m\sigma^2/36}$  with a probability larger than $1-\varepsilon$.
\end{lemma}
\begin{proof}
Denote $J= \supp \epsilon_0$. We need to ensure that
\begin{align*}
	\sprod{\nu, Ae_i} &> t , \quad i\in J\\
	\sprod{\nu, Ae_i} &< t,\quad i \notin J \sg{.}
\end{align*} 

\sg{First observe that,} for $i$ arbitrary, \sg{we obtain}
\begin{align*}
	 \sprod{\nu, A e_i} & = \sprod{\rho \one + D\epsilon_0- m^{-1} \sprod{D\epsilon_0, \one}\one, \mu \one + De_i} = \rho \mu m + \rho \sprod{\one, De_i} +\sprod{D\epsilon_0, De_i}- m^{-1} \sprod{D\epsilon_0,\one}\sprod{\one, De_i}\\&\sg{=:\rho \mu m+\rho X_1(i)+X_2(i)+m^{-1}X_3(i)}.
\end{align*} 
Now let us investigate each of the probabilistic terms \sg{$X_1(i),X_2(i)$ and $X_3(i)$} for \sg{the} cases $i \in J$ and $i \notin J$ separately.

For $i\notin J$ and $\sprod{\one, De_i}$ we have
\begin{align*}
	X_1(i)= \sprod{\one, De_i} = \sum_{\ell=1}^m d_{i\ell}.
\end{align*}
This is a sum of $m$ bounded, centered random variables. Hoeffding's inequality \cite[Theorem 7.20]{FR} implies that this variable is smaller than $-\theta$ only with a probability smaller than
\begin{align*}
	 \exp \left( - \frac{2\theta^2}{4m\Lambda^2}\right).
\end{align*}
Similarly, we may write the other terms as sums of $km$ (and $km^2$, respectively) centered, in $[-\Lambda^2, \Lambda^2]$ bounded random variables. For instance,\sg{we have}
\begin{align*}
\sg{X_2(i)=}\sprod{D\epsilon_0, De_i} = \sum_{j \in J} \sum_{\ell=1}^m d_{j\ell} d_{i\ell},
\end{align*}
and \sg{we can} similarly \sg{compute} $\sg{X_3(i)}=\sprod{D\epsilon_0, \one}\sprod{\one, D e_i}$. By applying Hoeffding's inequality twice, subsequently a union bound, and choosing the $\theta$-parameters equal to $\mu m /4$, $\rho \mu m /4$ and $\rho m^2 \mu/4$ respectively, we obtain that, for all $i \notin J$,
\begin{align*}
	  \sprod{\nu, Ae_i} \leq \frac{\rho \mu m}{4}
\end{align*}
with a failure probability at most
\begin{align*}
	 (n-\abs{J})\left( \exp \left( - \frac{ m\mu^2}{32\Lambda^2}\right) + \exp \left( - \frac{m\rho^2 \mu^2 }{32\abs{J}\Lambda^4}\right)+ \exp \left( - \frac{m^2\rho^2  \mu^2}{32\abs{J}\Lambda^4}\right) \right).
\end{align*}


The arguments for the case $i \in J$ are very similar, with the only exception that the variables $\sprod{D\epsilon_0, De_i}$ and $\sprod{D\epsilon_0, \one}\sprod{\one, D e_i}$ are not centered, but rather
\begin{align*}
	\erw{ \sprod{D\epsilon_0, De_i}} = \sum_{j \in J} \sum_{\ell=1}^m \erw{d_{j\ell} d_{i\ell}} = \sum_{j \in J} \sum_{\ell=1}^m \delta_{ij}\erw{d_{ij}^2} = m \sigma^2
\end{align*}
and similarly $\erw{\sprod{D\epsilon_0, \one}\sprod{\one, D e_i}} = m \sigma^2$.

%
%
By applying Hoeffding's inequality, setting the $\theta$-parameters equal to $\mu m$, $\rho \mu m$ and $\rho \mu m^2$, respectively, and a union bound, we obtain
\begin{align*}
 \sprod{\nu, Ae_i} \geq \rho \mu m + \rho \mu m + m \sigma^2 +\rho \mu m - m^{-1}(m \sigma^2 - m^2 \rho \mu) = m( \sigma^2(1-m^{-1}) + 4 \rho \mu ), \quad i \in J
\end{align*}
with a total failure probability no larger than
\begin{align*}
|J|\left( \exp\left(-\frac{m\mu^2}{2\Lambda^2}\right) + \exp\left(-\frac{m\rho^2\mu^2}{ 2\abs{J} \Lambda^4}\right)+  \exp\left(-\frac{m^2 \rho^2\mu^2}{ 2\abs{J}\Lambda^4}\right)\right). 
\end{align*}

\sg{Now with $\rho = -\mu^{-1}\sigma^2/8$ and using the assumption $m \geq 3$,  we} obtain
\begin{align*}
 \sprod{\nu, Ae_i} &\leq -\frac{m\sigma^2 }{32}, \quad i \notin J, \\
\sprod{\nu, Ae_i} &  \geq m( \sigma^2(1/2 - m^{-1}) \geq m \sigma^2/6,  i \in J, 
\end{align*}
with a total failure probability no larger than
\begin{align*}
n\left( \exp \left( - \frac{ m\mu^2}{32\Lambda^2}\right) + \exp \left( - \frac{m\sigma^4 }{2048\abs{J}\Lambda ^4}\right)+ \exp \left( - \frac{m^2\sigma^4}{2048\abs{J}\Lambda^4}\right) \right).
\end{align*}
\sg{Notice that this probability is smaller} than $\varepsilon$, provided
\begin{align*}
	m \gtrsim \max\left(\frac{\Lambda^2}{\mu^2}, \abs{J} \frac{\Lambda^4}{\sigma^4}\right)\log \left(\frac{n}{\varepsilon}\right).
\end{align*}
\end{proof}

The former lemma shows that the vector $\nu$ defined in \eqref{eq:dualCertificate} fulfills the first assumption of Proposition \ref{eq:StabCert}. We still have to control the norm of the vector $\nu$. For this, we prove the following lemma.

\begin{lemma} \label{lem:certNorm}	Let $A\in \R^{m,N}$ be a biased measurement matrix of the form \eqref{BiasedMM} with $\erw{d_{ij}^2}=\sigma^2$ and $d_{ij}\in [-\Lambda, \Lambda]$, for $i\in [m], j\in [N]$. Further, let $\nu\in \R^m$ be defined as in \eqref{eq:dualCertificate} and \sg{$\varepsilon>0$}.
	Under the assumption $$m \gtrsim \left(\frac{\Lambda}{\sigma} \log(\varepsilon^{-1})^\frac{3}{4}\right)^{\frac{4}{3}},$$ \sg{we have}  $\norm{\nu}_2 \leq m(\rho^2 + \sigma^2 (\abs{J}+ 1/\abs{J}))$ with a probability larger than $1-\varepsilon$.
\end{lemma}

\begin{proof}
	\sg{First, observe that}
	\begin{align*}
		\norm{\nu}_2^2 = m\rho^2 + \sprod{D\epsilon_0, D\epsilon_0} - m^{-1} \abs{\sprod{D\epsilon_0, \one}}^2 \leq m \rho^2 + \sprod{D\epsilon_0,D\epsilon_0}
	\end{align*}
		The term $m \rho^2$ is constant and must therefore not be subject to any  probabilistic considerations. As for the term \sg{$\sprod{D\epsilon_0,D\epsilon_0}$}, we calculate
		\begin{align*}
			\sprod{D \epsilon_0, D\epsilon_0} = \sum_{j,k \in J} \sprod{De_j,De_k} = \sum_{j,k \in J} \sum_{\ell=1}^m d_{j\ell}d_{k\ell}.
		\end{align*}
		This is the sum of $\abs{J}^2m$ random variables, all bounded in $[-\Lambda^2, \Lambda^2]$. Again, Hoeffding's inequality implies
		\begin{align*}
			\prb{ \sprod{D \epsilon_0, D\epsilon_0}- \erw{\sprod{D \epsilon_0, D\epsilon_0}}<-t} \leq \exp\left(- \frac{2\abs{J}^2m t^2}{4\Lambda^4}\right).
		\end{align*}
	Let us also note that $\erw{\sprod{D \epsilon_0, D\epsilon_0}}	= \abs{J}m \sigma^2$ (since $\erw{d_{j\ell}d_{k\ell}} = \sigma^2 \delta_{ 	jk}$.

	\sg{Summarizing}, we obtain
	\begin{align*}
		\norm{\nu}^2  \leq m\rho^2 + \abs{J}m \sigma^2 + t_1
	\end{align*}
	with a failure probability smaller than
	\begin{align*}
	  \exp\left(- \frac{2\abs{J}m t_1^2}{4\Lambda^4} \right),
	\end{align*}
for some $t_1>0$.
	We choose $t_1 = m\sigma^2/(4\abs{J})$ to obtain the final failure probability
	\begin{align*}
		\exp\left(- \frac{2m^3 \sigma^4}{16\cdot 4\Lambda^4} \right)
	\end{align*}
	We see that provided $ m \gtrsim \left(\frac{\Lambda}{\sigma} \log(\epsilon^{-1})^\frac{3}{4}\right)^{\frac{4}{3}} $, we \sg{obtain}
	\begin{align*}
		\norm{\nu}_2^2 \leq m(\rho^2 + \Lambda^2 (\abs{J}+ 1/\abs{J}))
	\end{align*}
	\sg{with a probability} larger than $1-\epsilon$, which was to be proven.
\end{proof}

We can now state and prove \sg{the second main result Theorem \ref{thm:simple2} more precisely}.

\begin{theorem}\label{thm:phaseTrans}
	Let $A\in \R^{m,N}$ be a biased measurement matrix of the form \eqref{BiasedMM} with $\erw{d_{ij}^2}=\sigma^2$ and $d_{ij}\in [-\Lambda, \Lambda]$, for $i\in [m], j\in [N]$. Fix some tolerance $\varepsilon>0$. A binary signal $x_0\sg{\in\{0,1\}^N}$ with $\|x_0\|_0=k$  is the unique solution of \eqref{Pbin} and \eqref{LSbin} for $b=Ax_0$ with probability larger than $1-\varepsilon$, provided
\begin{align}
m \gtrsim \max\left(\frac{\Lambda^2}{\mu^2}, \abs{\min(k, N-k)} \frac{\Lambda^4}{\sigma^4}\right)\log \left(\frac{N}{\varepsilon}\right)
\end{align}
with a constant depending only on $\sigma$ and $\mu^{-1}$. 

Under the additional assumption $m \gtrsim \left(\frac{\Lambda}{\sigma}\right)^{4/3}\log(\varepsilon^{-1})$ the solution $x_*$ of \eqref{LSbin} for $b= Ax_0 + n$ with $\norm{n}_2 \leq \epsilon$ \sg{for some $\epsilon >0$} obeys
\begin{align*}
\norm{x_0-x_*}_2 \leq \sqrt{\frac{9\left(\frac{16\sigma^2}{\mu^2}+ \min(k, N-k)\right)}{m \sigma^2}}\cdot \epsilon.
\end{align*}
\end{theorem}

\begin{proof}
	Let $\kappa= \min (k, N-k)$, and choose $\epsilon_0$ as the sparser of the two vectors $x_0$ and $\one-x_0$. By applying  Lemma \ref{lem:certHk}, we obtain a vector $\nu$ which either itself (in the case $\epsilon_0=x_0$) or its negative $-\nu$ (in the case $\epsilon_0 = \one - x_0$) lies in $H_k^t$, for some $t>0$, making Proposition \ref{prop:HKPlus} applicable. This concludes the proof of the first part.
	
	As for the second one, the additional assumption makes Lemma \ref{lem:certNorm} applicable, which \sg{provides} us \sg{with} the bound on the norm of $\nu$ needed to apply Proposition \ref{eq:StabCert}. Remembering that we have to choose $\rho= 4\sigma^2/\mu$ to \sg{be allowed to apply} \ref{lem:certHk}, we arrive at the stated estimate.
\end{proof}

	\label{sec:Main}

\subsection*{Acknowledgements}
Axel Flinth acknowledges support by the Deutsche Forschungsgemeinschaft (DFG) Grant KU 1446/18-1. \sg{Sandra Keiper acknowledges support by the DFG Collaborative Research
Center TRR 109 “Discretization in Geometry and Dynamics”. Both authors acknowledge support by the Berlin Mathematical School.}

\bibliographystyle{amsplain}
\bibliography{Bernoulli}

\end{document}